\documentclass[reqno,a4paper]{amsart}
\usepackage{amssymb}
\usepackage{amsmath}
\newcommand{\angles}[1]{\langle #1 \rangle}
\newcommand{\half}{\frac{1}{2}}
\newcommand{\abs}[1]{\vert #1 \vert}
\newcommand{\norm}[1]{\left\Vert #1 \right\Vert}
\newcommand{\R}{\mathbb{R}}

\begin{document} 
\newtheorem{prop}{Proposition}[section]
\newtheorem{Def}{Definition}[section]
\newtheorem{theorem}{Theorem}[section]
\newtheorem{lemma}{Lemma}[section]
 \newtheorem{Cor}{Corollary}[section]

\title[Yang-Mills in Lorenz gauge]{\bf Low regularity well-posedness for the Yang-Mills system in Fourier-Lebesgue spaces}
\author[Hartmut Pecher]{
{\bf Hartmut Pecher}\\
Fakult\"at f\"ur  Mathematik und Naturwissenschaften\\
Bergische Universit\"at Wuppertal\\
Gau{\ss}str.  20\\
42119 Wuppertal\\
Germany\\
e-mail {\tt pecher@math.uni-wuppertal.de}}
\date{}

\begin{abstract}
The Cauchy problem for the Yang-Mills system in three space dimensions with data in Fourier-Lebesgue spaces $\widehat{H}^{s,r}$ , $1 < r \le 2$ , is shown to be locally well-posed, where we have to assume only almost optimal minimal regularity for the data with respect to scaling as $r \to 1$ . This is true despite of the fact that no null condition is known for one of the critical quadratic nonlinearities, which prevented by now the corresponding result in the classical case $r=2$ with data in standard Sobolev spaces.
\end{abstract}
\maketitle
\renewcommand{\thefootnote}{\fnsymbol{footnote}}
\footnotetext{\hspace{-1.5em}{\it 2010 Mathematics Subject Classification:} 
35Q40, 35L70 }
\normalsize 
\setcounter{section}{0}

\section{Introduction}

\noindent 
Let $\mathcal{G}$ be the Lie group $SO(n,\mathbb{R})$ (the group of orthogonal matrices of determinant 1) or $SU(n,\mathbb{C})$ (the group of unitary matrices of determinant 1) and $g$ its Lie algebra $so(n,\mathbb{R})$ (the algebra of trace-free skew symmetric matrices) or $su(n,\mathbb{C})$ (the algebra of trace-free skew hermitian matrices) with Lie bracket $[X,Y] = XY-YX$ (the matrix commutator). 
For given  $A_{\alpha}: \mathbb{R}^{1+n} \rightarrow g $ we define the curvature $F=F[A]$ by
\begin{equation}
\label{curv}
 F_{\alpha \beta} = \partial_{\alpha} A_{\beta} - \partial_{\beta} A_{\alpha} + [A_{\alpha},A_{\beta}] \, , 
\end{equation}
where $\alpha,\beta \in \{0,1,...,n\}$ and $D_{\alpha} = \partial_{\alpha} + [A_{\alpha}, \cdot \,]$ .

Then the Yang-Mills system is given by
\begin{equation}
\label{1}
D^{\alpha} F_{\alpha \beta}  = 0
\end{equation}
in Minkowski space $\mathbb{R}^{1+n} = \mathbb{R}_t \times \mathbb{R}^n_x$ , where $n \ge 3$, with metric $diag(-1,1,...,1)$. Greek indices run over $\{0,1,...,n\}$, Latin indices over $\{1,...,n\}$, and the usual summation convention is used.  
We use the notation $\partial_{\mu} = \frac{\partial}{\partial x_{\mu}}$, where we write $(x^0,x^1,...,x^n)=(t,x^1,...,x^n)$ and also $\partial_0 = \partial_t$.

Setting $\beta =0$ in (\ref{1}) we obtain the Gauss-law constraint
\begin{equation}
\nonumber
\partial^j F_{j 0} + [A^j,F_{j0} ]=0 \, .
\end{equation}

The total energy for Yang-Mills at time $t$  is given by
$$
  \mathcal E(t) =\sum_{0\le \alpha, \beta\le n} \int_{\R^n} \abs{F_{\alpha \beta}(t,x)}^2  \, dx,
$$
and is conserved for a smooth solution decaying sufficiently fast at spatial infinity. 
The  Yang-Mills system is invariant with respect to the scaling
$$ A_{\lambda}(t,x) = \lambda A(\lambda t,\lambda x) \quad, \quad F_{\lambda}(t,x)= \lambda^2 F(\lambda t,\lambda x) \, . $$
This implies
\begin{align*}
\|A_{\lambda}(0,\cdot)\|_{\dot{\widehat{H}}^{s,r}} = \lambda^{1+s-\frac{n}{r}} \|a_{\lambda}\|_{\dot{\widehat{H}}^{s,r}} \, , \\
\|F_{\lambda}(0,\cdot)\|_{\dot{\widehat{H}}^{l,r}} = \lambda^{2+l-\frac{n}{r}} \|f_{\lambda}\|_{\dot{\widehat{H}}^{l,r}}	\, .
\end{align*}
Here $\|u\|_{\widehat{H}^{s,r}} := \| \langle \xi \rangle^s \widehat{u}(\xi)\|_{L^{r'}}$ , where $r$ and $r'$ are dual exponents, and $\hat{\widehat{H}}^{s,r}$ denotes the homogeneous space.
Therefore the scaling critical exponent is $s = \frac{n}{r}-1$ for $A$ and $l= \frac{n}{r}-2$ for $F$ .

The system is gauge invariant. Given a sufficiently smooth function $U: {\mathbb R}^{1+n} \rightarrow \mathcal{G}$ we define the gauge transformation $T$ by $T A_0 = A_0'$ , 
$T(A_1,...,A_n) = (A_1',...,A_n'),$ where
\begin{align*}
A_{\alpha} & \longmapsto A_{\alpha}' = U A_{\alpha} U^{-1} - (\partial_{\alpha} U) U^{-1}  \, . 
\end{align*}

It is  well-known that if  $(A_0,...A_n)$ satisfies (\ref{curv}),(\ref{1}) so does $(A_0',...,A_n')$.

Hence we may impose a gauge condition. We exclusively study the case $n=3$ and Lorenz gauge $\partial^{\alpha}A_{\alpha} =0$. Other convenient gauges are the Coulomb gauge $\partial^j A_j=0$ and the temporal gauge $A_0 =0$.
Our aim is to obtain local well-posedness for data with minimal regularity. Up to now there exists no result for data arbitrarily close to the critical scaling regularity.

The classical case $r=2$ with data in standard Sobolev spaces was considered by Klainerman and Machedon \cite{KM}, who made the decisive detection that
the nonlinearity satisfies a so-called null condition, which enabled them to prove global well-posedness in temporal and in Coulomb gauge in energy space. The corresponding result in Lorenz gauge, where the Yang-Mills equations can be formulated as a system of nonlinear wave equations, was shown by Selberg and Tesfahun \cite{ST}, who discovered 
that also in this case some of the nonlinearities have a null structure. Tesfahun \cite{T} improved this result to data without finite energy, namely for $(A(0),(\partial_t A)(0)) \in H^s \times H^{s-1}$ and $(F(0),(\partial_t F)(0)) \in H^l \times H^{l-1}$ with $s = \frac{6}{7}+\epsilon$ and $l = -\frac{1}{14}+\epsilon$ for any $\epsilon > 0$ by discovering an additional partial null structure. A further improvement was achieved by the author \cite{P}, namely to $(s,l)=(\frac{5}{7}+\epsilon,-\frac{1}{7}+\epsilon)$ by modifying the solution spaces appropriately.

 As the critical case with respect to scaling is $(s,l)=(\half,-\half)$ , there is however still a gap, a phenomenon, which is also present in other gauges. The present paper closes this gap in the sense that as $r \to 1$ we almost reach the critical case $(s,l) = (2,1)$ .

Local well-posedness in energy space was also given by Oh \cite{O} using a new gauge, namely the Yang-Mills heat flow. He was also able to show that this solution can be globally extended \cite{O1}.  The Cauchy problem was also treated in higher space dimensions by several authors (\cite{KS},\cite{KT},\cite{KrT},\cite{KrSt},\cite{P1}).

In the present paper we treat the local well-posedness problem for the Yang-Mills system in Lorenz gauge and space dimension $n = 3$ in the case of data $(A(0),(\partial_t A)(0)) \in \widehat{H}^{s,r} \times \widehat{H}^{s-1,r}$ and $(F(0),(\partial_t F)(0)) \in \widehat{H}^{l,r} \times \widehat{H}^{l-1,r}$ in Fourier-Lebesgue spaces for $r \neq 2$, which coincide with the classical Sobolev spaces $H^s$ for $r=2$. The assumption is that $s = \frac{16}{7r} - \frac{2}{7}+ \delta$ and $ l =\frac{15}{7r}-\frac{8}{7}+ \delta $ , where any  $\delta > 0$ is admissible.  Thus we obtain $s \to 2+\delta$ and $l \to 1+\delta$ as $r \to 1$ , which is almost optimal with respect to scaling.

Such an approach was used by several authors already, starting with Vargas-Vega \cite{VV} for 1D Schr\"odinger equations. Gr\"unrock showed LWP for the modified KdV equation \cite{G}, a result which was improved by Gr\"unrock and Vega \cite{GV}. Gr\"unrock treated derivative nonlinear wave equations in 3+1 dimensions \cite{G1} and obtained an almost optimal result as $r \to 1$ with respect to scaling. Systems of nonlinear wave equations in the 2+1 dimensional case for nonlinearities which fulfill a null condition were considered by Grigoryan-Nahmod \cite{GN}. The latter two results are based on estimates by Foschi and Klainerman \cite{FK}.

In Chapter 2 we start by giving the formulation of the Yang-Mills equations as a system of semilinear wave equations and formulate the main theorem (Theorem \ref{Theorem1.1} and Cor. \ref{Cor}). In chapter 3 we recall some basic facts about our solution spaces and a general local well-posedness theorem for the Cauchy problem for systems of nonlinear wave equations with data in Fourier-Lebesgue spaces, which allows to reduce it to estimates for the nonlinearities. In chapter 4 we give the final formulation of the system in terms of null forms as far as possible. The bi-, tri- and quadrilinear estimates sufficient for the local well-posedness result are formulated, where we rely on Tesfahun's paper \cite{T}. In chapter 5 we prove bilinear estimates for the null forms and for general bilinear terms in generalized Bourgain-Klainerman-Machedon spaces $H^r_{s,b}$ (and $X^r_{s,b,\pm}$) based on elementary but non-trivial estimates by Foschi and Klainerman \cite{FK}, Gr\"unrock \cite{G} and Grigoryan-Nahmod \cite{GN}. In the final chapter 6 we prove the multilinear estimates formulated in chapter 4 by reduction to the bilinear estimates of chapter 5.

{\bf Acknowledgement:} I cordially thank the referees. Their reports helped very much to improve the paper and especially to fill in two gaps in it.

\section{Main results}
Expanding (\ref{1}) in terms of the gauge potentials $\left\{A_\alpha\right\}$, we obtain: 
 \begin{equation}\label{YM2}
  \square A_\beta = \partial_\beta\partial^\alpha A_\alpha- [\partial^\alpha A_\alpha, A_\beta] - [A^\alpha,\partial^\alpha A_\beta] - 
  [A^\alpha, F_{\alpha\beta}].
\end{equation}
If we now impose the Lorenz gauge condition,
the system \eqref{YM2} reduces to the nonlinear wave equation
 \begin{equation}\label{YM3}
  \square A_\beta = - [A^\alpha,\partial_\alpha A_\beta] -   [A^\alpha, F_{\alpha\beta}].
\end{equation}
In addition, regardless of the choice of gauge, $F$ satisfies the wave equation 
    \begin{equation}\label{YMF1}
 \begin{split}
   \square F_{\beta\gamma}&=-[A^\alpha,\partial_\alpha F_{\beta\gamma}] - \partial^\alpha[A_\alpha,F_{\beta\gamma}] - \left[A^\alpha,[A_\alpha,F_{\beta\gamma}]\right]
   \\
   & \quad - 2[F^{\alpha}_{{\;\;\,}\beta},F_{\gamma\alpha}] \, ,
   \end{split}
 \end{equation}
where we refer to \cite{ST}, chapter 3.2.

Expanding the second and fourth terms in \eqref{YMF1}, and also imposing the Lorenz gauge, yields 
\begin{equation}\label{YMF2}
\begin{split}
       \square F_{\beta\gamma}&= - 2[A^\alpha,\partial_\alpha F_{\beta\gamma}]
      + 2[\partial_\gamma A^\alpha, \partial_\alpha A_\beta]
      - 2[\partial_\beta A^\alpha, \partial_\alpha A_\gamma]
      \\
      &\quad + 2[\partial^\alpha A_\beta , \partial_\alpha A_\gamma]
                 + 2[\partial_\beta A^\alpha, \partial_\gamma A_\alpha] - [A^\alpha,[A_\alpha,F_{\beta\gamma}]] 
                 \\
                 &\quad  +2[F_{\alpha\beta},[A^\alpha,A_\gamma]]- 2[F_{\alpha\gamma},[A^\alpha,A_\beta]]
                      - 2[[A^\alpha,A_\beta],[A_\alpha,A_\gamma]] .  
                         \end{split}    
\end{equation}

 Note on the other hand by expanding the last term in the right hand side of \eqref{YM3}, we obtain 
 \begin{equation}\label{YM4}
  \square A_\beta = - 2[A^\alpha,\partial_\alpha A_\beta] + [A^\alpha,\partial_\beta A_\alpha] - 
  [A^\alpha, [A_\alpha,A_\beta]].
\end{equation}

We want to solve the system \eqref{YMF2}-\eqref{YM4} simultaneously for $A$ and $F$.
So to pose the Cauchy problem for this system, we consider initial data for $(A,F)$ at $t=0$:
\begin{equation}\label{Data-AF}
A(0) = a, \quad \partial_t A(0) = \dot a,
        \quad
    F(0) =f, \quad \partial_t F(0) = \dot f.
 \end{equation}

In fact, the initial data for $F$ can be determined from $(a, \dot a)$ as follows:
\begin{equation}\label{f}
\left\{
\begin{aligned}
  f_{ij} &= \partial_i a_j - \partial_j a_i + [a_i,a_j],
  \\
  f_{0i} &= \dot a_i - \partial_i a_0 + [a_0,a_i],
\\
  \dot f_{ij} &= \partial_i \dot a_j - \partial_j \dot a_i + [\dot a_i,a_j]+[ a_i, \dot a_j],
  \\
 \dot f_{0i} &= \partial^j f_{ji} +[a^\alpha, f_{\alpha i}] 
\end{aligned}
\right.
\end{equation}
where the first three expressions come from \eqref{curv} whereas 
the last one comes from (\ref{1}) with $\beta=i$.

Note that the Lorenz gauge condition $\partial^\alpha A_\alpha=0$ and (\ref{1}) with $\beta=0$ impose the constraints 
\begin{equation}\label{Const}
\dot a_0= \partial^i a_i,
\quad
   \partial^i f_{i0} + [a^i, f_{i0}] = 0 \, .
\end{equation}

Now we formulate our main theorem.
\begin{theorem}
\label{Theorem1.1}
Let $1 < r \le 2$ , $\delta > 0$. Assume that $s$ and $l$ satisfy the following conditions:
$$s = \frac{16}{7r} - \frac{2}{7}+ \delta\, ,  \quad l =\frac{15}{7r}-\frac{8}{7}+ \delta\, . $$
Given initial data $(a,\dot{a}) \in \widehat{H}^{s,r} \times \widehat{H}^{s-1,r}$ , $(f,\dot{f}) \in \widehat{H}^{l,r} \times \widehat{H}^{l-1,r}$ ,  there exists a time $T > 0$ , $T=T(\|a\|_{\widehat{H}^{s,r}},\|\dot{a}\|_{\widehat{H}^{s-1,r}} , \|f\|_{\widehat{H}^{l,r}} , \|\dot{f}\|_{\widehat{H}^{l-1,r}})$, such that the Cauchy problem (\ref{YMF2}),(\ref{YM4}),(\ref{Data-AF}) has a unique solution $A_{\mu}\in X^r_{s,b,+}[0,T]+ X^r_{s,b,-}[0,T]$ , $F \in X^r_{l,a,+}[0,T]+ X^r_{l,a,-}[0,T]$ (these spaces are defined in Def. \ref{Def.}). Here $a = \frac{1}{r}+$ and $b= \half + \frac{1}{2r}+$ . This solution has the regularity
$$ A_{\mu} \in C^0([0,T],\widehat{H}^{s,r}) \cap C^1([0,T],\widehat{H}^{s-1,r}) \, , \, F \in C^0([0,T],\widehat{H}^{l,r}) \cap C^1([0,T],\widehat{H}^{l-1,r}) \, . $$
The solution depends continuously on the data and persistence of higher regularity holds.
\end{theorem}

\begin{Cor}
\label{Cor}
Let $s,r$ fulfill the assumptions of Theorem \ref{Theorem1.1}. Moreover assume that the initial data fulfill (\ref{f}) and (\ref{Const}). Given any $(a,\dot{a}) \in \widehat{H}^{s,r} \times \widehat{H}^{s-1,r}$ , there exista a time $T=T(\|a\|_{\widehat{H}^{s,r}},\|\dot{a}\|_{\widehat{H}^{s-1,r}} , \|f\|_{\widehat{H}^{l,r}} , \|\dot{f}\|_{\widehat{H}^{l-1,r}})$ , such that the solution $(A,F)$ of Theorem \ref{Theorem1.1} satisfies the Yang-Mills system (\ref{curv}),(\ref{1}) with Cauchy data $(a,\dot{a})$ and the Lorenz gauge condition $\partial^{\alpha} A_{\alpha} =0$ .
\end{Cor}
\begin{proof}[Proof of the Corollary]
 If $(a,\dot{a}) \in \widehat{H}^{s,r} \times \widehat{H}^{s-1,r}$ , then $(f,\dot{f})$ , defined by (\ref{f}), fulfill $(f,\dot{f}) \in \widehat{H}^{l,r} \times \widehat{H}^{l-1,r}$, as one easily checks. Thus we may apply Theorem \ref{Theorem1.1}.
The solution $(A,F)$ does not necessarily fulfill the Lorenz gauge condition and (\ref{curv}), i.e. $F=F[A]$ . If however the conditions (\ref{f}) and (\ref{Const}) are assumed then these properties are satisfied and $(A,F)$ is a solution of the Yang-Mills system (\ref{curv}),(\ref{1}) with Cauchy data $(a,\dot{a})$. This was shown in \cite{ST}, Remark 2. 
\end{proof}

Let us fix some notation.
We denote the Fourier transform with respect to space and time  by $\,\,\widehat{}\,$ . 
 $\Box = \partial_t^2 - \Delta$ is the d'Alembert operator,
$a\pm := a \pm \epsilon$ for a sufficiently small $\epsilon >0$ , and $\langle \,\cdot\, \rangle := (1+|\cdot|^2)^{\frac{1}{2}}$ . \\
Let $\Lambda^{\alpha}$
be the multiplier with symbol  $
\langle\xi \rangle^\alpha $ . Similarly let $D^{\alpha}$,
and $D_{-}^{\alpha}$ be the multipliers with symbols $
\abs{\xi}^\alpha$ and $\quad ||\tau|-|\xi||^\alpha$ ,
respectively.

\begin{Def}
\label{Def.}
Let $1\le r\le 2$ , $s,b \in \R$ . The wave-Sobolev spaces $H^r_{s,b}$ are the completion of the Schwarz space ${\mathcal S}(\R^{1+3})$ with norm
$$ \|u\|_{H^r_{s,b}} = \| \langle \xi \rangle^s \langle  |\tau| - |\xi| \rangle^b \widehat{u}(\tau,\xi) \|_{L^{r'}_{\tau \xi}} \, , $$ 
where $r'$ is the dual exponent to $r$.
We also define $H^r_{s,b}[0,T]$ as the space of the restrictions of functions in $H^r_{s,b}$ to $[0,T] \times \mathbb{R}^3$.  Similarly we define $X^r_{s,b,\pm} $ with norm  $$ \|\phi\|_{X^r_{s,b\pm}} := \| \langle \xi \rangle^s \langle \tau \pm |\xi| \rangle^b \tilde{\phi}(\tau,\xi)\|_{L^{r'}_{\tau \xi}} $$ and $X^r_{s,b,\pm}[0,T] $ .
$\dot{H}^r_{s,b}$ and $\dot{X}^r_{s,b,\pm} $ are the corresponding homogeneous spaces,  where $\langle \xi \rangle$ is replaced by $|\xi|$ .
\end{Def}

\section{Preliminaries}

We start by collecting some fundamental properties of the solution spaces. We rely on \cite{G}. The spaces $X^r_{s,b,\pm} $ with norm  $$ \|\phi\|_{X^r_{s,b\pm}} := \| \langle \xi \rangle^s \langle \tau \pm |\xi| \rangle^b \tilde{\phi}(\tau,\xi)\|_{L^{r'}_{\tau \xi}} $$ for $1<r<\infty$ are Banach spaces with ${\mathcal S}$ as a dense subspace. The dual space is $X^{r'}_{-s,-b,\pm}$ , where $\frac{1}{r} + \frac{1}{r'} = 1$. The complex interpolation space is given by
$$(X^{r_0}_{s_0,b_0,\pm} , X^{r_1}_{s_1,b_1,\pm})_{[\theta]} = X^r_{s,b,\pm} \, , $$
where $s=(1-\theta)s_0+\theta s_1$, $\frac{1}{r} = \frac{1-\theta}{r_0} + \frac{\theta}{r_1}$ , $b=(1-\theta)b_0 + \theta b_1$ . Similar properties has the space $H^r_{s,b}$ .\\
If $u=u_++u_-$, where $u_{\pm} \in X^r_{s,b,\pm} [0,T]$ , then $u \in C^0([0,T],\hat{H}^{s,r})$ , if $b > \frac{1}{r}$ .

The "transfer principle" in the following proposition, which is well-known in the case $r=2$, also holds for general $1<r<\infty$ (cf. \cite{GN}, Prop. A.2 or \cite{G}, Lemma 1). We denote $ \|u\|_{\hat{L}^p_t(\hat{L}^q_x)} := \|\tilde{u}\|_{L^{p'}_{\tau} (L^{q'}_{\xi})}$ .
\begin{prop}
\label{Prop.0.1}
Let $1 \le p,q \le \infty$ .
Assume that $T$ is a bilinear operator which fulfills
$$ \|T(e^{\pm_1 itD} f_1, e^{\pm_2itD} f_2)\|_{\hat{L}^p_t(\hat{L}^q_x)} \lesssim \|f_1\|_{\hat{H}^{s_1,r}} \|f_2\|_{\hat{H}^{s_2,r}}$$
for all combinations of signs $\pm_1,\pm_2$ , then for $b > \frac{1}{r}$ the following estimate holds:
$$ \|T(u_1,u_2)\|_{\hat{L}^p_t(\hat{L}^q_x)} \lesssim \|u_1\|_{H^r_{s_1,b}}  \|u_2\|_{H^r_{s_2,b}} \, . $$
\end{prop}

The general local well-posedness theorem is the following (obvious generalization of)  \cite{G}, Thm. 1.
\begin{theorem}
\label{Theorem0.3}
Let $N_{\pm}(u,v):=N_{\pm}(u_+,u_-,v_+,v_-)$ and $M_{\pm}(u,v):=M_{\pm}(u_+,u_-,\\v_+,v_-)$ be multilinear functions.
Assume that for given $s,l \in \R$, $1 < r < \infty$ there exist $ b,a > \frac{1}{r}$ such that the estimates
$$ \|N_{\pm}(u,v)\|_{X^r_{s,b-1+,\pm}} \le \omega_1( \|u\|_{X^r_{s,b}},\|v\|_{X^r_{l,a}}) $$
and 
$$\|M_{\pm}(u,v)\|_{X^r_{l,a-1+,\pm}} \le \omega_2( \|u\|_{X^r_{s,b}},\|v\|_{X^r_{l,a}}) $$
are valid with nondecreasing functions $\omega_j$ , where $\|u\|_{X^r_{s,b}} := \|u_-\|_{X^r_{s,b,-}} + \|u_+\|_{X^r_{s,b,+}}$. Then there exist $T=T(\|u_{0_ {\pm}}\|_{\hat{H}^{s,r}},\|v_{0_{\pm}}\|_{\hat{H}^{l,r}})$ $>0$ and a unique solution $(u_+,u_-,\\v_+,v_-) \in X^r_{s,b,+}[0,T] \times X^r_{s,b,-}[0,T] \times X^r_{l,a,+}[0,T] \times X^r_{l,a,-}[0,T] $ of the Cauchy problem
$$ \partial_t u_{\pm} \pm i\Lambda u = N_{\pm}(u,v) \quad , \quad \partial_t v_{\pm} \pm i\Lambda v = M_{\pm}(u,v) $$ $$         u_{\pm}(0) = u_{0_{\pm}} \in \hat{H}^{s,r} \quad , \quad v_{\pm}(0) = v_{0_{\pm}} \in \hat{H}^{l,r}       \, . $$
 This solution is persistent and the mapping data upon solution $(u_{0+},u_{0-},v_{0+},v_{0-}) \\ \mapsto (u_+,u_-,v_+,v_-)$ , $\hat{H}^{s,r} \times \hat{H}^{s,r}\times \hat{H}^{l,r} \times \hat{H}^{l,r} \to X^r_{s,b,+}[0,T_0] \times X^r_{s,b,-}[0,T_0] \times X^r_{l,a,+}[0,T_0]\times X^r_{l,a,-}[0,T_0] $ is locally Lipschitz continuous for any $T_0 < T$.
\end{theorem}

\section{Reformulation of the problem and null structure}
The reformulation of the Yang-Mills equations and the reduction of our main theorem to nonlinear estimates is completely taken over from Tesfahun \cite{T} (cf. also the fundamental paper by Selberg and Tesfahun \cite{ST}).

The standard null forms are given by
\begin{equation}\label{OrdNullforms}
\left\{
\begin{aligned}
Q_{0}(u,v)&=\partial_\alpha u \partial^\alpha v=-\partial_t u \partial_t v+\partial_i u \partial^j v,
\\
Q_{\alpha\beta}(u,v)&=\partial_\alpha u \partial_\beta v-\partial_\beta u \partial_\alpha v.
\end{aligned}
        \right.
\end{equation}
For $ g$-valued $u,v$, define a commutator version of null forms by 
\begin{equation}\label{CommutatorNullforms}
\left\{
\begin{aligned}
  Q_0[u,v] &= [\partial_\alpha u, \partial^\alpha v] = Q_0(u,v) - Q_0(v,u),
  \\
  Q_{\alpha\beta}[u,v] &= [\partial_\alpha u, \partial_\beta v] - [\partial_\beta u, \partial_\alpha v] = Q_{\alpha\beta}(u,v) + Q_{\alpha\beta}(v,u).
\end{aligned}
\right.
\end{equation}

 Note the identity
\begin{equation}\label{NullformTrick}
  [\partial_\alpha u, \partial_\beta u]
  = \frac12 \left( [\partial_\alpha u, \partial_\beta u] - [\partial_\beta u, \partial_\alpha u] \right)
  = \frac12 Q_{\alpha\beta}[u,u].
\end{equation}

Define 
\begin{equation}\label{NewNull} 
  \mathcal{Q}[u,v] = - \frac12 \varepsilon^{ijk}\varepsilon_{klm} Q_{ij}\left[R^l u^m, v \right]
  - Q_{0i}\left[R^i u_0, v \right],
\end{equation}
where $\varepsilon_{ijk}$ is the antisymmetric symbol with $\varepsilon_{123} = 1$ and
$R_i = \Lambda^{-1}\partial_i $ are the Riesz transforms.

Now we refer to Tesfahun \cite{T}, who showed that the system (\ref{YMF2}),(\ref{YM4}) in Lorenz gauge can be written in the following form
\begin{equation}\label{AF}
\begin{aligned}
  \square A_\beta &=  \mathcal M_\beta(A,\partial_t A,F,\partial_t F),
  \\
  \square F_{\beta\gamma} &=  \mathcal N_{\beta\gamma}(A,\partial_t A,F,\partial_t F),
\end{aligned}
\end{equation}
where
\begin{align*}
  \mathcal M_\beta(A,\partial_t A,F,\partial_t F) &= -2 \mathcal Q[\Lambda^{-1} A,A_\beta] +
  \sum_{i=1}^4\Gamma^i_\beta(A, \partial A, F, \partial F)-2[\Lambda^{-2}  A^\alpha, \partial_\alpha A_\beta ]
  \\
 &\quad  - [A^\alpha, [A_\alpha, A_\beta]],
   \end{align*}

  \begin{align*}
  \mathcal N_{ij}(A,\partial_t A,F,\partial_t F)
  = &- 2\mathcal Q[\Lambda^{-1} A,F_{ij}]
  + 2\mathcal Q[\Lambda^{-1} \partial_j A, A_i]- 2\mathcal Q[\Lambda^{-1} \partial_i A, A_j] 
  \\
  & + 2Q_0[A_i , A_j]
  + Q_{ij}[A^\alpha,A_\alpha]-2[\Lambda^{-2}  A^\alpha, \partial_\alpha F_{ij} ]
  \\
  &+2[\Lambda^{-2}  \partial_jA^\alpha, \partial_\alpha A_{i} ]-2[\Lambda^{-2}  \partial_i A^\alpha, \partial_\alpha A_{j} ]
  \\
  & - [A^\alpha,[A_\alpha,F_{ij}]] + 2[F_{\alpha i},[A^\alpha,A_j]] - 2[F_{\alpha j},[A^\alpha,A_i]]
  \\
  & - 2[[A^\alpha,A_i],[A_\alpha,A_j]],
  \end{align*} 
  
  \begin{align*}
  \mathcal N_{0i}(A,\partial_t A,F,\partial_t F)
  = &- 2\mathcal Q[\Lambda^{-1} A,F_{0i}]
  + 2\mathcal Q[\Lambda^{-1} \partial_i A, A_0]- 2 Q_{0j}[A^j,A_i] \\
	&+ 2Q_0[A_0 , A_i]
  + Q_{0i}[A^\alpha,A_\alpha]-2[\Lambda^{-2}  A^\alpha, \partial_\alpha F_{0i} ]\\
	&+2[\Lambda^{-2}  \partial_i A^\alpha, \partial_\alpha A_{0} ]
   - [A^\alpha,[A_\alpha,F_{0i}]] + 2[F_{\alpha 0},[A^\alpha,A_i]] \\ 
	&- 2[F_{\alpha i},[A^\alpha,A_0]]
   - 2[[A^\alpha,A_0],[A_\alpha,A_i]]
  \end{align*}
  where
\begin{equation}\label{Gammas}
\begin{aligned}
\Gamma^1_\beta(A, \partial A,  F, \partial F)
&=-[A_0,\partial_\beta A_0] + 
[\Lambda^{-1} R_j (\partial_t A_0),  \Lambda^{-1} R^j  \partial_t (\partial_\beta A_0)] ,
\\
\Gamma^2_\beta(A, \partial A,  F, \partial F)
&= -\frac12 \varepsilon^{ijk} 
\varepsilon_{klm} \Big\{ Q_{ij}[\Lambda^{-1}  R^n \mathbf A_n ,\Lambda^{-1} R^l \partial_\beta\mathbf A^m ] \\
&\hspace{6em}+ 
Q_{ij}[\Lambda^{-1}  R^n \partial_\beta \mathbf A_n , \Lambda^{-1}  R^l \mathbf A^m ]
 \Big\},
\\
\Gamma^3_\beta(A, \partial A,  F, \partial F)&=[ \Lambda^{-2} \nabla \times \mathbf F,
\Lambda^{-2} \nabla \times \mathbf \partial_\beta \mathbf F]
\\
&\quad
-[\Lambda^{-2}\nabla \times \mathbf F , 
\Lambda^{-2} \partial_\beta \nabla \times( \mathbf A \times \mathbf A)]
\\
& \quad- [ \Lambda^{-2} \nabla \times( \mathbf A \times \mathbf A),
\Lambda^{-2}\nabla \times \mathbf \partial_\beta \mathbf F ] 
\\
&\quad+ [\Lambda^{-2}\nabla \times( \mathbf A \times \mathbf A),
\Lambda^{-2}\partial_\beta \nabla \times( \mathbf A \times \mathbf A) ],
\\
\Gamma^4_\beta(A, \partial A,  F, \partial F)&=[\mathbf A^{\text{cf}} + \mathbf A^{\text{df}}, \Lambda^{-2} \partial_\beta \mathbf A] 
+ [\Lambda^{-2}\mathbf A,  \partial_\beta \mathbf A  ].
 \end{aligned}
\end{equation}
Here $\mathbf F = (F_1,F_2,F_3)$ , where $F_i = \sum_{j<k \,,\, j,k \neq i} F_{jk}$ ,
$(\nabla \times \mathbf A)_i = \epsilon_{ijk} \partial^j A^k$ and $(\mathbf A \times \mathbf B)_k = \epsilon_{ijk} A^i B^j$ .

Here especially the
splitting of the spatial part $\mathbf A=(A_1,A_2, A_3)$ of the potential into divergence-free and curl-free parts and a smoother part is used
\begin{equation}\label{SplitA}  
\mathbf A = \mathbf A^{\text{df}} + \mathbf A^{\text{cf}} + \Lambda^{-2} \mathbf A,
\end{equation}
where
\begin{align*}
  \mathbf A^{\text{df}}&=  \Lambda^{-2} \nabla \times \nabla \times \mathbf A,
  \\
  \mathbf A^{\text{cf}}&= -\Lambda^{-2} \nabla (\nabla \cdot \mathbf A).
\end{align*}

In a standard way we rewrite the system (\ref{AF}) as a first order (in t) system. Defining
$A_{\pm} = \half(A \pm (i\Lambda)^{-1} \partial_t A) \quad , \quad F_{\pm} = \half(F \pm (i \Lambda)^{-1}\partial_t F) $ , so that $A=A_++A_-$ , $\partial_t A= i \Lambda(A_+-A_-)$ , $F=F_++F_-$ , $\partial_t F = i \Lambda (F_+-F_-)$ the system transforms to
\begin{align}
(i \partial_t \pm \Lambda)A^{\beta}_{\pm} & = -A^{\beta} \mp(2 \Lambda)^{-1} \mathcal M_\beta(A,\partial_t A,F,\partial_t F)\, ,\\
(i \partial_t \pm \Lambda)F^{\beta \gamma}_{\pm} & = -F^{\beta \gamma} \mp(2 \Lambda)^{-1} \mathcal N_{\beta\gamma}(A,\partial_t A,F,\partial_t F) \, .
\end{align}
The initial data transform to 
$$A_{\pm}(0)= \half(a \pm (i\Lambda)^{-1} \dot{a}) \in 
\widehat{H}^{s,r} \quad , \quad F_{\mp}(0)= \half(f \pm (i\Lambda)^{-1} \dot{f}) \in \widehat{H}^{l,r} \, .$$

Now, looking at the terms in $ \mathcal{M}_{\beta}$ and $ \mathcal{N}_{\beta \gamma}$ and noting the fact that the Riesz transforms 
$R_i$ are bounded in the spaces involved, the estimates in Theorem \ref{Theorem0.3} 
reduce to proving:\\
1. the estimates for the null forms $Q_{ij}$ , $Q_0$ and $Q \in \{Q_{0i},Q_{ij}\}$ :
\begin{align}
  \label{21}
  \norm{ Q[\Lambda^{-1} A, A]}_{H^r_{s-1,b-1+}}
  &\lesssim \|A\|_{X^r_{s,b}} \|A\|_{X^r_{s,b}},
  \\
    \label{22}
  \norm{  Q_{ij}[\Lambda^{-1} A, \Lambda^{-1} \partial A]}_{H^r_{s-1,b-1+}}
  &\lesssim  \|A\|_{X^r_{s,b}} \|A\|_{X^r_{s,b}} ,
  \\
 \label{23}
  \norm{ Q[\Lambda^{-1}A, F]}_{H^r_{l-1,a-1+} }
  &\lesssim \|A\|_{X^r_{s,b}}  \|F\|_{X^r_{l,a}},
    \\
  \label{24}
  \norm{  Q[ A,   A]}_{H^r_{l-1,a-1+} }
  &\lesssim \|A\|_{X^r_{s,b}} \|A\|_{X^r_{s,b}} ,\\
	\label{25}
  \norm{ Q_0[ A,   A]}_{H^r_{l-1,a-1+} }
  &\lesssim \|A\|_{X^r_{s,b}}  \|A\|_{X^r_{s,b}} ,
	\end{align} 
the following estimate for $\Gamma^1$ and other bilinear terms
\begin{align}
  \label{26}
  \norm{\Gamma^1( A, \partial A)}_{H^r_{s-1,b-1+}}
  &\lesssim\|A\|_{X^r_{s,b}} \|A\|_{X^r_{s,b}} ,
  \\
  \label{27} 
   \norm{\Pi( A, \Lambda^{-2} \partial A  ) }_{H^r_{s-1,b-1+}}
     &\lesssim \|A\|_{X^r_{s,b}} \|A\|_{X^r_{s,b}} ,
     \\
     \label{28} 
   \norm{ \Pi( \Lambda^{-2} A,  \partial A)   }_{H^r_{s-1,b-1-+}}
     &\lesssim \|A\|_{X^r_{s,b}} \|A\|_{X^r_{s,b}} ,
     \\
  \label{29} 
  \norm{\Pi(\Lambda^{-1} F, \Lambda^{-1} \partial F  ) }_{H^r_{s-1,b-1+}}
     &\lesssim \|F\|_{X^r_{l,a}} \|F\|_{X^r_{l,a}},
     \\
       \label{30} 
   \norm{\Pi( \Lambda^{-2} A,  \partial F)   }_{H^r_{l-1,a-1+}}
     &\lesssim \|A\|_{X^r_{s,b}} \|F\|_{X^r_{l,a}},
     \\
       \label{31} 
   \norm{ \Pi( \Lambda^{-1} A,  \partial A)   }_{H^r_{l-1,a-1+}}
     &\lesssim \|A\|_{X^r_{s,b}} \|A\|_{X^r_{s,b}}
     \end{align}
and\\
2. the following trilinear and quadrilinear estimates:
 \begin{align}
   \label{32}
   \norm{\Pi(\Lambda^{-1} F,\Lambda^{-1} \partial( AA) )}_{H^r_{s-1,b-1+}}
  &\lesssim  \|F\|_{X^r_{l,a}} \|A\|_{X^r_{s,b}} \|A\|_{X^r_{s,b}} ,
  \\
   \label{33}
   \norm{\Pi(\Lambda^{-1}\partial F, \Lambda^{-1}  ( AA) )}_{H^r_{s-1,b-1+}}
  &\lesssim  \|F\|_{X^r_{l,a}} \|A\|_{X^r_{s,b}} \|A\|_{X^r_{s,b}} ,
  \\
  \label{34}
   \norm{\Pi(\Lambda^{-1}(AA), \Lambda^{-1} \partial ( AA)) }_{H^r_{s-1,b-1+}}
  &\lesssim  \|A\|_{X^r_{s,b}} \|A\|_{X^r_{s,b}} \|A\|_{X^r_{s,b}} \|A\|_{X^r_{s,b}} ,
  \\
   \label{35} 
  \norm{\Pi(A,A,A)}_{H^r_{s-1,b-1+}}
  &\lesssim\|A\|_{X^r_{s,b}} \|A\|_{X^r_{s,b}} \|A\|_{X^r_{s,b}},
  \\
   \label{36}
  \norm{\Pi(A, A, F)}_{H^r_{l-1,a-1+}}
  &\lesssim \|A\|_{X^r_{s,b}} 
  \|A\|_{X^r_{s,b}}\|F\|_{X^r_{l,a}},
  \\
   \label{37}
  \norm{\Pi(A,A, A, A)}_{H^r_{l-1,a-1+}}
  &\lesssim  \|A\|_{X^r_{s,b}} \|A\|_{X^r_{s,b}} \|A\|_{X^r_{s,b}} \|A\|_{X^r_{s,b}} \, .
     \end{align}
$\Pi(\cdots)$ denotes a multilinear operator in its arguments and
$\|u\|_{X^r_{s,b}} := \|u_-\|_{X^r_{s,b,-}} + \|u_+\|_{X^r_{s,b,+}}$ .

The matrix commutator null forms are linear combinations of the ordinary ones, 
in view of \eqref{CommutatorNullforms}. Since the matrix 
structure plays no role in the estimates under consideration, 
we reduce (\ref{21})--(\ref{25}) to estimates of the ordinary null forms for $\mathbb C$-valued 
functions $u$ and $v$ (as in \eqref{OrdNullforms}).

Next we consider the term $\Gamma_{\beta}^1$ . We may ignore its matrix form and treat 
$$\Gamma^1_k(A_0,, \partial_k A_0)
=-A_0 (\partial_k A_0) + 
\Lambda^{-1} R_j (\partial_t A_0)  \Lambda^{-1} R^j  \partial_t (\partial_k A_0))$$
for $k=1,2,3$ and
\begin{align*}
\Gamma^1_0(A_0,, \partial^i A_i)
&=-A_0 (\partial_0 A_0) + 
\Lambda^{-1} R_j (\partial_t A_0)  \Lambda^{-1} R^j  \partial_t (\partial_0 A_0)) \\
& = -A_0 (\partial^i A_i) + 
\Lambda^{-1} R_j (\partial_t A_0)  \Lambda^{-1} R^j  \partial_t (\partial^i A_i)) \,,
\end{align*}
where we used the Lorenz gauge $\partial_0 A_0 = \partial^i A_i$ in the last line in order to eliminate one time derivative. Thus we have to consider
$$ \Gamma^1(u,v) = -uv + \Lambda^{-1} R_j (\partial_t u) \Lambda^{-1} R^j(\partial_t v) \, , $$
where $u=A_0$ and $v=\partial^i A_i$ or $v=\partial_k A_0$ .

Next we show that $\Gamma^1$ also has a null structure.
The proof of the following lemma was essentially given by Tesfahun \cite{T}. In fact the detection of this null structure was the main progress of his paper over Selberg-Tesfahun \cite{ST}.  
\begin{lemma}
\label{Lemma2.1}
Let $q_{\mu \nu}(u,v) := Q_{\mu \nu}(D^{-1}u,D^{-1}v)$ , $q_0(u,v) := Q_0(D^{-1}u,D^{-1}v)$ .
The following estimate holds:
\begin{align}
\label{45'}
\Gamma^1(u,v) &\precsim \sum_{i,j} q_{ij}(u,v) + q_0(u,v) + (\Lambda^{-2}u)v + u(\Lambda^{-2}v) \, .  
\end{align}
Here $u \preceq v$ means $\abs{\widehat u} \lesssim \abs{\widehat v}$ .
\end{lemma}
\begin{proof}
$\Gamma^1(u,v)$ has the symbol
\begin{align*}
p(\xi,\tau,\eta,\lambda)& = -1 + \frac{\angles{\xi,\eta} \tau \lambda}{\angles{\xi}^2 \angles{\eta}^2} = \left( -1 + \frac{\angles{\xi,\eta} \angles{\xi,\eta}}{\angles{\xi}^2 \angles{\eta}^2}\right) + \frac{(\tau \lambda -\angles{\xi,\eta}) \angles{\xi,\eta}}{\angles{\xi}^2 \angles{\eta}^2}  = I + II
\end{align*}
Now we estimate
\begin{align*}
|I| & = \left| \frac{|\xi|^2 |\eta|^2\cos^2 \angle(\xi,\eta)}{\angles{\xi}^2 \angles{\eta}^2} -1 \right| \\
&
\le \left| \frac{\angles{\xi}^2 \angles{\eta}^2 \cos^2 \angle(\xi,\eta)}{\angles{\xi}^2 \angles{\eta}^2} -1 \right|
+ \left| \frac{|\xi|^2 |\eta|^2 - \angles{\xi}^2 \angles{\eta}^2}{\angles{\xi}^2 \angles{\eta}^2} \right| \\
& = \sin^2 \angle(\xi,\eta) + \left| \frac{|\xi|^2 |\eta|^2 - \angles{\xi}^2 \angles{\eta}^2}{\angles{\xi}^2 \angles{\eta}^2} \right| \, ,
\end{align*}
where $\angle(\xi,\eta)$ denotes the angle between $\xi$ and $\eta$ .
We have
$$\left| \frac{|\xi|^2 |\eta|^2 - \angles{\xi}^2 \angles{\eta}^2}{\angles{\xi}^2 \angles{\eta}^2} \right| = \frac{|\xi|^2 + |\eta|^2 + 1}{\angles{\xi}^2 \angles{\eta}^2} \le \frac{1}{\angles{\xi}^2} + \frac{1}{\angles{\eta}^2} $$
and
$$\sin \angle(\xi,\eta) = \frac{\xi \times \eta}{|\xi|\,|\eta|} \, . $$
 Thus the operator belonging to the symbol I is controlled by 
$\sum_{i,j} q_{ij}(u,v) +  (\Lambda^{-2}u)v + u(\Lambda^{-2}v)$ .
Moreover 
$$ |II| \le \frac{|\tau \lambda - \angles{\xi,\eta}|}{\angles{\xi} \angles{\eta}} \le |q_0(\xi,\eta)| \, .$$ 
Thus we obtain (\ref{45'}).
\end{proof}

\section{Bilinear estimates}
The proof of the following bilinear estimates relies on estimates given by Foschi and Klainerman \cite{FK}. We first treat the case $r>1$ , but close to $1$.
\begin{lemma}
\label{Lemma1}
Assume $0 \le\alpha_1,\alpha_2 $ ,  $\alpha_1+\alpha_2 \ge \frac{2}{r}$ and $ b > \frac{1}{r}$. The following estimate applies
$$ \|q_{ij}(u,v)\|_{H^r_{0,0}} \lesssim \|u\|_{X^r_{\alpha_1,b,\pm_1}} \|v\|_{X^r_{\alpha_2,b,\pm_2}} \, . $$
\end{lemma}
\begin{proof}
	Because we use inhomogeneous norms it is obviously possible to assume $\alpha_1 + \alpha_2 = \frac{2}{r}$ . Moreover, by interpolation we may reduce to the case $\alpha_1= \frac{2}{r}$ , $\alpha_2 =0$ .
	
The left hand side of the claimed estimate equals
$$ \|{\mathcal F}(q_{ij}(u,v))\|_{L^{r'}_{\tau \xi}} = \| \int q_{ij}(\eta,\eta-\xi) \tilde{u}(\lambda,\eta) \tilde{v}(\tau - \lambda,\xi - \eta) d\lambda d\eta \|_{L^{r'}_{\tau \xi}} \, . $$
Let now $u(t,x) = e^{\pm_1 iD} u_0^{\pm_1}(x)$ , $v(t,x) = e^{\pm_2 iD} v_0^{\pm_2}(x)$ , so that 
$$ \tilde{u}(\tau,\xi) = c \delta(\tau \mp_1 |\xi|) \widehat{u_0^{\pm_1}}(\xi) \quad , \quad \tilde{v}(\tau,\xi) = c \delta(\tau \mp_2 |\xi|) \widehat{v_0^{\pm_2}}(\xi) \, . $$
This implies
\begin{align*}
&\|{\mathcal F}(q_{ij}(u,v))\|_{L^{r'}_{\tau \xi}} \\
&= c^2 \| \int q_{ij}(\eta,\eta-\xi) \widehat{u_0^{\pm_1}}(\eta) \widehat{v_0^{\pm_2}}(\xi-\eta) \,\delta(\lambda \mp_1 |\eta|) \delta(\tau-\lambda\mp_2|\xi-\eta|) d\lambda d\eta \|_{L^{r'}_{\tau \xi}} \\
& = c^2 \| \int q_{ij}(\eta,\eta-\xi) \widehat{u_0^{\pm_1}}(\eta) \widehat{v_0^{\pm_2}}(\xi-\eta) \,\delta(\tau\mp_1|\eta| \mp_2|\xi-\eta|) d\eta \|_{L^{r'}_{\tau \xi}} \, .
\end{align*}
By symmetry we only have to consider the  elliptic case $\pm_1=\pm_2 = +$ and the hyperbolic case $\pm_1= + \, , \, \pm_2=-$ .  \\
{\bf Elliptic case.} We obtain by \cite{FK}, Lemma 13.2:
$$|q_{ij}(\eta,\xi-\eta)| \le \frac{|\eta \times (\xi - \eta)|}{|\eta| \, |\xi - \eta|} \lesssim \frac{|\xi|^{\half} (|\eta| + |\xi - \eta| - |\xi|)^{\half}}{|\eta|^{\half} |\xi - \eta|^{\half}} \, . $$
By H\"older's inequality we obtain
\begin{align*}
&\|{\mathcal F}(q_{ij}(u,v))\|_{L^{r'}_{\tau \xi}} \\
& \lesssim \|\int \frac{|\xi|^{\half} ||\tau|-|\xi||^{\half}}{|\eta|^{\half} |\xi - \eta|^{\half}} \,
 \delta(\tau-|\eta|-|\xi - \eta|) \, |\widehat{u_0^+}(\eta)| \, |\widehat{v_0^+}(\xi - \eta)| d\eta \|_{L^{r'}_{\tau \xi}} \\
& \lesssim \sup_{\tau,\xi} I  \,\, \|\widehat{D^{\frac{2}{r}} u_0^+}\|_{L^{r'}} \| \widehat{v_0^+}\|_{L^{r'}} \, ,
\end{align*}
where
$$ I = |\xi|^{\half} ||\tau|-|\xi||^{\half} \left( \int \delta(\tau - |\eta| - |\xi - \eta|) \, |\eta|^{-(\frac{2}{r} + \half)r} |\xi - \eta|^{-\frac{r}{2}} d\eta \right)^{\frac{1}{r}} \, . $$
We want to prove $ \sup_{\tau,\xi} I \lesssim 1 $ . By \cite{FK}, Lemma 4.3 we obtain
$$\int \delta(\tau - |\eta| - |\xi - \eta|) \, |\eta|^{-(\frac{2}{r} + \half)r} |\xi - \eta|^{-\frac{r}{2}} d\eta \sim \tau^A ||\tau|-|\xi||^B \, , $$
where $A= \max(\frac{2}{r}+\half)r,2) - (\frac{2}{r}+1)r= -\frac{r}{2}$ and $B=2-\max(\frac{2}{r}+\half)r,2)=-\frac{r}{2}$ . This implies
$$
I  \lesssim |\xi|^{\half} ||\tau|-|\xi||^{\half} \tau^{-\half} ||\tau|-|\xi||^{-\half}\le  1 $$
using $|\xi| \le |\tau|$ . \\
{\bf Hyperbolic case.} We start with the following bound (cf. \cite{FK}, Lemma 13.2):
$$  |q_{ij}(\eta,\xi-\eta)| \le \frac{|\eta \times (\xi-\eta)|}{|\eta|\,|\xi-\eta|} \lesssim \frac{|\xi|^{\half} (|\xi|-||\eta|-|\eta-\xi||)^{\half}}{|\eta|^{\half} |\xi-\eta|^{\half}} \, , $$
so that similarly as in the elliptic case we have to estimate
$$ I = |\xi|^{\half} ||\tau|-|\xi||^{\half} \left( \int \delta(\tau - |\eta| + |\xi - \eta|) \, |\eta|^{-(\frac{2}{r} + \half)r} |\xi - \eta|^{-\frac{r}{2}} d\eta \right)^{\frac{1}{r}} \, . $$
In the subcase $|\eta|+|\xi-\eta| \le 2|\xi|$ we apply \cite{FK}, Prop. 4.5 and obtain
$$\int \delta(\tau - |\eta| + |\xi - \eta|) \, |\eta|^{-(\frac{2}{r} + \half)r} |\xi - \eta|^{-\frac{r}{2}} d\eta \sim |\xi|^A ||\xi|-|\tau||^B \, . $$
where $A=\max((\frac{2}{r}+\half)r,2) - (\frac{2}{r}+1)r = -\frac{r}{2}$ and $B= 2- \max((\frac{2}{r}+\half)r,2)= -\frac{r}{2}$. \\
This implies
$$I \lesssim |\xi|^{\half} ||\tau|-|\xi||^{\half} |\xi|^{-\half} ||\tau|-|\xi||^{-\half} = 1 \, . $$
In the subcase $|\eta| + |\xi-\eta| \ge 2|\xi|$ we obtain by \cite{FK}, Lemma 4.4:
 \begin{align*}
&\int \delta(\tau - |\eta| + |\xi - \eta|) \, |\eta|^{-2-\frac{r}{2}} |\xi - \eta|^{-\frac{r}{2}} d\eta \\
&\sim \int_2^{\infty} (|\xi|x+\tau)^{-1-\frac{r}{2}} (|\xi|x-\tau)^{1-\frac{r}{2}} dx \\
&\sim \int_2^{\infty} (x+\frac{\tau}{|\xi|})^{-1-\frac{r}{2}} (x-\frac{\tau}{|\xi|})^{1-\frac{r}{2}} dx \, \cdot|\xi|^{-r} \, .
\end{align*}
We remark that in fact the lower limit of the integral can be chosen as 2 by inspection of the proof in \cite{FK}.
The integral converges, because $|\tau| \le |\xi|$ and $r > 1.$  This implies the bound
$$ I \lesssim |\xi|^{\half} ||\tau|-|\xi||^{\half} |\xi|^{-1} \lesssim  1 \, . $$
Summarizing we obtain
$$\|q_{ij}(u,v)\|_{X^r_{0,0}} \lesssim \|D^{\frac{2}{r}} u_0^{\pm_1}\|_{L^{r'}}  \| v_0^{\pm_2}\|_{L^{r'}} \, . $$
By the transfer principle Prop. \ref{Prop.0.1} we obtain the claimed result. 
\end{proof}

In a similar manner we can also estimate the nullform $q_{0j}(u,v)$ .
\begin{lemma}
\label{Lemma2}
Assume $0 \le \alpha_1,\alpha_2 $ ,  $\alpha_1+\alpha_2 \ge \frac{2}{r}$ and $ b > \frac{1}{r}$ . The following estimate applies
$$ \|q_{0j}(u,v)\|_{H^r_{0,0}} \lesssim \|u\|_{X^r_{\alpha_1,b,\pm_1}} \|v\|_{X^r_{\alpha_2,b,\pm_2}} \, . $$
\end{lemma}
\begin{proof}
	Again we may  reduce to the case $\alpha_1= \frac{2}{r}$ and $\alpha_2=0$ .
Arguing as in the proof of Lemma \ref{Lemma1} we use in the elliptic case the estimate (cf. \cite{FK}, Lemma 13.2):
$$|q_{0j}(\eta,\xi-\eta)| \lesssim \frac{(|\eta|+|\xi-\eta|-|\xi|)^{\half}}{\min(|\eta|^{\half},|\xi-\eta|^{\half})} \, . $$
In the case  $|\eta| \le |\xi-\eta|$ we obtain
\begin{align*}
I &= ||\tau|-|\xi||^{\half} \left( \int \delta(\tau - |\eta| - |\xi - \eta|) \, |\eta|^{-2-\frac{r}{2}} d\eta \right)^{\frac{1}{r}} \\
&\sim ||\tau|-|\xi||^{\half} |\tau|^{\frac{A}{r}}  ||\tau|-|\xi||^{\frac{B}{r}} = 1\, ,
\end{align*}
because $ A=\max((\frac{2}{r}+\half)r,2) - (\frac{2}{r}+\half)r = 0$ and $B= -\frac{r}{2}$ . \\
In the case $|\eta| \ge |\xi-\eta|$ we obtain
\begin{align*}
I &= ||\tau|-|\xi||^{\half} \left( \int \delta(\tau - |\eta| - |\xi - \eta|) \, |\eta|^{-2} |\xi-\eta|^{-\frac{r}{2}} d\eta \right)^{\frac{1}{r}} \\
&\sim ||\tau|-|\xi||^{\half} |\tau|^{\frac{A}{r}}  ||\tau|-|\xi||^{\frac{B}{r}} (1+\log \frac{|\tau|}{||\tau|-|\xi||})^{\frac{1}{r}}\, ,
\end{align*}
where we are in the exceptional case, where a logarithmic factor appears, and $A= \max(\frac{r}{2},2)-2-\frac{r}{2} = -\frac{r}{2}$ and $B=2-\max(\frac{r}{2},2) =0$, so that
$$I \lesssim ||\tau|-|\xi||^{\half} \tau^{-\half} (1+ \frac{|\tau|^{\epsilon}}{||\tau|-|\xi||^{\epsilon}}) \lesssim  1 \, .$$
 In the hyperbolic case we obtain by \cite{FK}, Lemma 13.2:
$$|q_{0j}(\eta, \xi-\eta)| \lesssim |\xi|^{\half} \frac{(|\xi|-||\eta|-|\eta-\xi||)^{\half}}{|\eta|^{\half} |\xi-\eta|^{\half}} $$
and argue exactly as in the proof of Lemma \ref{Lemma1}. The proof is completed as before.
\end{proof}

We also need the same result for $q_0(u,v)$ .
\begin{lemma}
\label{Lemma3}
Assume $0 \le \alpha_1,\alpha_2 $ , $\alpha_1+\alpha_2 \ge \frac{2}{r}$ and $ b > \frac{1}{r}$ . The following estimate applies
$$ \|q_0(u,v)\|_{H^r_{0,0}} \lesssim \|u\|_{X^r_{\alpha_1,b,\pm_1}} \|v\|_{X^r_{\alpha_2,b,\pm_2}} \, . $$
\end{lemma}
\begin{proof}
	As before we reduce to the case $\alpha_1= \frac{2}{r}$ and $\alpha_2 =0$ .
We use in the elliptic case the estimate (cf. \cite{FK}, Lemma 13.2):
$$|q_{0}(\eta,\xi-\eta)| \lesssim \frac{(|\eta|+|\xi-\eta|-|\xi|)}{\min(|\eta|,|\xi-\eta|)} \, . $$
In the case $|\eta| \le |\xi-\eta|$ we have to estimate
\begin{align*}
I &= ||\tau|-|\xi|| \left( \int \delta(\tau - |\eta| - |\xi - \eta|) \, |\eta|^{-2-r}  d\eta \right)^{\frac{1}{r}} \\
&\sim ||\tau|-|\xi|| \, |\tau|^{\frac{A}{r}}  ||\tau|-|\xi||^{\frac{B}{r}}= 1 \, ,
\end{align*}
because $A=\max(2+r,2)-(2+r)=0$ and $B=2-\max(2+r,2)= -r$ . \\
In the case  $|\eta| \ge |\xi-\eta|$ we obtain
\begin{align*}
I &= ||\tau|-|\xi|| \left( \int \delta(\tau - |\eta| - |\xi - \eta|) \, |\eta|^{-2}  |\xi-\eta|^{-r} d\eta \right)^{\frac{1}{r}} \\
&\sim ||\tau|-|\xi|| \, |\tau|^{\frac{A}{r}}  ||\tau|-|\xi||^{\frac{B}{r}} (1+ \log \frac{|\tau|}{||\tau|-|\xi||})^\frac{1}{r}  \, .
\end{align*}
Here we are in the exceptional case, where the logarithmic factor appears.  Because $A=\max(2,r)-2-r=-r$ and $B=2-\max(2,r)= 0$ we obtain by $|\xi| \le |\tau|$ the estimate
$$ I \sim   ||\tau|-|\xi|| \, |\tau|^{-1}   (1+ \log \frac{|\tau|}{||\tau|-|\xi||})^\frac{1}{r} \lesssim 1 \,. $$
In the hyperbolic case we obtain by \cite{FK}, Lemma 13.2:
$$|q_{0j}(\eta, \xi-\eta)| \lesssim |\xi|\frac{|\xi|-||\eta|-|\eta-\xi||}{|\eta|\, |\xi-\eta|} \, .$$
In the subcase $|\eta|+|\xi-\eta| \le 2|\xi|$ we apply \cite{FK}, Prop. 4.5 and obtain
$$\int \delta(\tau - |\eta| + |\xi - \eta|) \, |\eta|^{-2-r} |\xi - \eta|^{-r} d\eta \sim |\xi|^A ||\xi|-|\tau||^B \ $$
with $A=\max(2+r,r,2)-2-2r=-r$ , $B=2-\max(2+r,r,2)=-r$ , so that
$$I \lesssim |\xi| \,||\tau|-|\xi|| \,|\xi|^{-1} ||\tau|-|\xi|^{-1} = 1 \, . $$
In the subcase $|\eta| + |\xi-\eta| \ge 2|\xi|$ we obtain by \cite{FK}, Lemma 4.4:
 \begin{align*}
&\int \delta(\tau - |\eta| + |\xi - \eta|) \, |\eta|^{-2-r} |\xi - \eta|^{-r} d\eta \\
&\sim \int_2^{\infty} (x+\frac{\tau}{|\xi|})^{-1-r} (x-\frac{\tau}{|\xi|})^{-r+1} dx \, \cdot|\xi|^{-2r} \, .
\end{align*}
The integral converges, because $|\tau| \le |\xi|$ . This implies the bound
$$ I \lesssim |\xi| \,||\tau|-|\xi||\, |\xi|^{-2} \lesssim  1 \, . $$
The proof is completed as the proof of Lemma \ref{Lemma1}.
\end{proof}

\begin{lemma}
\label{Lemma4}
Assume $0 \le \alpha_0 \le \alpha_1,\alpha_2 $ , $\alpha_1,\alpha_2 \neq \frac{2}{r}$ , $\alpha_1+\alpha_2-\alpha_0 =\frac{2}{r}$ , $\alpha_1+\alpha_2 >\frac{3}{r}$, $b > \frac{1}{r}$. Then the following estimate applies:
$$\|uv\|_{\dot{H}^r_{\alpha_0,0}} \lesssim \|u\|_{\dot{X}^r_{\alpha_1,b,\pm_1}} \|v\|_{\dot{X}^r_{\alpha_2,b,\pm_2}} \, . $$
\end{lemma}
\begin{proof}
By similar calculations as before we have to estimate in the elliptic case:
\begin{align*}
I & =|\xi|^{\alpha_0}  ( \int \delta(\tau-|\eta|-|\xi-\eta|) \eta|^{-\alpha_1r} |\xi-\eta|^{-\alpha_2 r} d\eta)^{\frac{1}{r}} \\
& \sim |\xi|^{\alpha_0} \tau^{\frac{A}{r}} ||\tau|-|\xi||^{\frac{B}{r}} \lesssim \tau^{\alpha_0+\frac{2}{r}-(\alpha_1+\alpha_2)} = 1 \, ,
\end{align*}
if we assume from now on $\alpha_1+\alpha_2-\alpha_0 = \frac{2}{r}$ and $\alpha_1,\alpha_2 < \frac{2}{r}$ ,
because by \cite{FK},Lemma 4.3 we obtain $A=\max(\alpha_1 r,\alpha_2 r,2)-(\alpha_1+\alpha_2)r  = 2-(\alpha_1+\alpha_2)r$ and $B= 2- \max(\alpha_1 r,\alpha_2 r,2)=0$ . 

In the hyperbolic case we obtain in the subcase $|\eta |+|\xi-\eta| \le 2|\xi|$
\begin{align*}
I &= (\int \delta(\tau-|\eta|+|\xi-\eta|) |\eta|^{-\alpha_1 r} |\xi-\eta|^{-\alpha_2 r} d\eta)^{\frac{1}{r}} |\xi|^{\alpha_0} \\
& \sim |\xi|^{\frac{A}{r}} ||\xi|-|\tau||^{\frac{B}{r}} |\xi|^{\alpha_0} \lesssim |\xi|^{\frac{2}{r}-(\alpha_1+\alpha_2)+\alpha_0} = 1 \, ,
\end{align*}
where we applied \cite{FK}, Prop. 4.5 with $A=\max(\alpha_1r,2) -(\alpha_1+\alpha_2)r = 2-(\alpha_1+\alpha_2)r$ and $B= 2-\max(\alpha_1 r,2) =0$ provided $\alpha_1< \frac{2}{r}$ and $\alpha_1+\alpha_2-\alpha_0 = \frac{2}{r}$ (or similarly with $\alpha_1$ and $\alpha_2$ interchanged). In the subcase $|\eta|+|\xi-\eta| \ge 2|\xi|$ we obtain by \cite{FK}, Lemma 4.4 the estimate
\begin{align*}
I & \sim (\int_2^{\infty} (|\xi|x+\tau)^{-\alpha_1 r +1} (|\xi|x-\tau)^{-\alpha_2 r+1} dx)^{\frac{1}{r}} |\xi|^{\alpha_0} \\
& \lesssim (\int_2^{\infty} (x+\frac{\tau}{|\xi|})^{-\alpha_1r +1}
(x-\frac{\tau}{|\xi|})^{-\alpha_2 r+1} dx)^{\frac{1}{r}} \, \cdot 
|\xi|^{\frac{2}{r}-(\alpha_1+\alpha_2)+\alpha_0} \lesssim 1 \, .
\end{align*}
Because $|\tau| \le |\xi|$ the integral converges provided $\alpha_1+\alpha_2 > \frac{3}{r}$ . Moreover we used again $\alpha_1+\alpha_2-\alpha_0 = \frac{2}{r}$ .

By the transfer principle the claimed estimate results as before.
\end{proof}

\begin{lemma}
\label{Lemma5}
If $\alpha_1,\alpha_2,b_1,b_2 \ge 0$ , $\alpha_1+\alpha_2 > \frac{3}{r}$ and $b_1+b_2 > \frac{1}{r}$ the following estimate applies:
$$ \|uv\|_{H^r_{0,0}} \lesssim \|u\|_{H^r_{\alpha_1,b_1}} \|v\|_{H^r_{\alpha_2,b_2}} \, . $$
\end{lemma}
\begin{proof}
By Young's and H\"older's inequalities we obtain
\begin{align*}
\|uv\|_{X^r_{0,0}} &= \|\widehat{uv}\|_{L^{r'}_{\tau \xi}} \lesssim \|\widehat{u}\|_{L^{p_1}_{\xi} L^{q_1}_{\tau}} \|\widehat{v}\|_{L^{p_2}_{\xi} L^{q_2}_{\tau}} \\
& \lesssim \| \langle \xi \rangle^{-\alpha_1} \langle|\tau|-|\xi|\rangle^{-b_1}\|_{L^{s_1}_{\xi} L^{r_1}_{\tau}} \| \langle \xi \rangle^{\alpha_1} \langle|\tau|-|\xi|\rangle^{b_1} \widehat{u}\|_{L^{r'}_{\xi \tau}} \, \cdot \\
& \; \;\cdot  \| \langle \xi \rangle^{-\alpha_2} \langle|\tau|-|\xi|\rangle^{-b_2}\|_{L^{s_2}_{\xi} L^{r_2}_{\tau}} \| \langle \xi \rangle^{\alpha_2} \langle|\tau|-|\xi|\rangle^{b_2} \widehat{v}\|_{L^{r'}_{\xi \tau}} \\
& \lesssim \|u\|_{H^r_{\alpha_1,b_1}} \|v\|_{H^r_{\alpha_2,b_2}} \, .
\end{align*}
Here $1+\frac{1}{r'} = \frac{1}{p_1}+\frac{1}{p_2} = \frac{1}{q_1}+\frac{1}{q_2}$ and $\frac{1}{q_j} = \frac{1}{r_j} + \frac{1}{r'}$ , $\frac{1}{p_j} = \frac{1}{s_j}+\frac{1}{r'} $ . This implies $1+\frac{1}{r'}= \frac{1}{s_1}+\frac{1}{s_2} + \frac{2}{r'} = \frac{1}{r_1}+\frac{1}{r_2}+\frac{2}{r'}$ . For the last estimate we need $r_j b_j > 1$ and $s_j \alpha_j > 3 $ . This can be fulfilled, if $1+\frac{1}{r'} < \frac{\alpha_1+\alpha_2}{3}+ \frac{2}{r'} \, \Leftrightarrow \, \alpha_1+\alpha_2 > \frac{3}{r}$ and $1+\frac{1}{r'} < b_1+b_2+\frac{2}{r'} \, \Leftrightarrow \, b_1+b_2 > \frac{1}{r} $ .
\end{proof}
An immediate consequence of Lemma \ref{Lemma4} and Lemma \ref{Lemma5} is the following corollary.
\begin{Cor}
\label{Cor.2}
Let $0 \le \alpha_0\le \alpha_1,\alpha_2 $ , $\alpha_1,\alpha_2 \neq \frac{2}{r}$ , $\alpha_1+\alpha_2-\alpha_0 \ge \frac{2}{r}$ , $\alpha_1+\alpha_2 > \frac{3}{r}$ and $b > \frac{1}{r}$. Then the following estimate applies:
$$\|uv\|_{H^r_{\alpha_0,0}} \lesssim \|u\|_{X^r_{\alpha_1,b,\pm_1}} \|v\|_{X^r_{\alpha_2,b,\pm_2}} \, . $$
\end{Cor}

\begin{lemma}
\label{Lemma6}
Let $0 \le \alpha_0 \le \alpha_1,\alpha_2$ and  $\alpha_1+\alpha_2-\alpha_0 \ge \frac{2}{r}+b$ . Then the following estimate applies:
$$ \|uv\|_{H^r_{\alpha_0,b}} \lesssim
 \|u\|_{X^r_{\alpha_1,b}} \|v\|_{X^r_{\alpha_2,b}} \, . $$
\end{lemma}
\begin{proof}
	By the fractional Leibniz rule we may assume $\alpha_0=0$ and by interpolation we only have to consider the case $\alpha_1=\frac{2}{r}+b$ , $\alpha_2=0$ .
We apply the "hyperbolic Leibniz rule" (cf. \cite{AFS}, p. 41):
\begin{equation}
\label{HLR}
 ||\tau|-|\xi|| \lesssim ||\rho|-|\eta|| + ||\tau - \rho| - |\xi-\eta|| + b_{\pm}(\xi,\eta) \, , 
\end{equation}
where
 $$ b_+(\xi,\eta) = |\eta| + |\xi-\eta| - |\xi| \quad , \quad b_-(\xi,\eta) = |\xi| - ||\eta|-|\xi-\eta|| \, . $$

Let us first consider the term $b_{\pm}(\xi,\eta)$ in (\ref{HLR}). Decomposing as before $uv=u_+v_++u_+v_-+u_-v_++u_-v_-$ , where $u_{\pm}(t)= e^{\pm itD} f , v_\pm(t) = e^{\pm itD} g$  , we use
$$ \widehat{u}_{\pm}(\tau,\xi) = c \delta(\tau \mp |\xi|) \widehat{f}(\xi) \quad , \quad  \widehat{v}_{\pm}(\tau,\xi) = c \delta(\tau \mp |\xi|) \widehat{g}(\xi)$$
and have to estimate
\begin{align*}
&\| \int  b^{b}_{\pm}(\xi,\eta) \delta(\tau - |\eta| \mp |\xi-\eta|) \widehat{f}(\xi) \widehat{g}(\xi-\eta)d\eta \|_{L^{r'}_{\tau \xi}} \\
& = \| \int ||\tau|-|\xi||^b \delta(\tau - |\eta| \mp |\xi-\eta|) \widehat{f}(\xi) \widehat{g}(\xi-\eta) d\eta \|_{L^{r'}_{\tau \xi}} \\
& \lesssim \sup_{\tau,\xi} I \, \|\widehat{D^{\frac{2}{r}+b} f}\|_{[L^{r'}}
\|\widehat{g}\|_{[L^{r'}} \, .
\end{align*}
Here we used H\"older's inequality, where
$$I = ||\tau|-|\xi||^b  (\int \delta(\tau-|\eta|\mp|\xi-\eta|) |\eta|^{-2-b r} d\eta)^{\frac{1}{r}} \, . $$
In order to obtain $I \lesssim 1$ we first consider the elliptic case $\pm_1=\pm_2=+$ and use \cite{FK}, Prop. 4.3. This 
$$ I \sim ||\tau|-|\xi||^b  \tau^{\frac{A}{r}} ||\tau|-|\xi||^{\frac{B}{r}} = ||\tau|-|\xi||^b ||\tau|-|\xi||^{-b} = 1$$
with $A=\max(2+br,2)-(2+br)=0$ and $B=2-\max(2+br,2)= -br$ .

Next we consider the hyperbolic case $\pm_1 = + \, , \, \pm_2=-$ . \\
First we assume $|\eta|+|\xi-\eta| \le 2 |\xi|$ and use \cite{FK}, Prop. 4.5 which gives
$$\int \delta(\tau - |\eta| + |\xi-\eta|) |\eta|^{-2-br}  d\eta \sim |\xi|^A ||\xi|-|\tau||^B \, , $$
where $A=\max(0,2)-(2+br)= -br$ , $B=2-\max(0,2)=0$ , if $0 \le \tau\le |\xi|$ , and $A=\max(2+br,2)-(2+br)=0$ , $B=2-\max(2+br,2)=-br$ , if $-|\xi| \le \tau \le 0$. In either case we easily obtain  $I \lesssim 1$ .\\
Next we assume $|\eta|+|\xi-\eta| \ge 2|\xi|$ , use \cite{FK}, Lemma 4.4 and obtain
\begin{align*}
&\int \delta(\tau-|\eta|-|\xi-\eta|) |\eta|^{-2-br} d\eta \\
& \sim \int_2^{\infty} (|\xi|x + \tau)^{-1-br} (|\xi|x-\tau) dx \\
& \sim \int_2^{\infty} (x+\frac{\tau}{|\xi|})^{-1-br} (x-\frac{\tau}{|\xi|}) dx \,\cdot |\xi|^{-br}  \, .
\end{align*}
This integral converges, because $\tau \le |\xi|$ and $b >\frac{1}{r}$ .This implies
$$I \lesssim  ||\tau|-|\xi||^b |\xi|^{-b} \lesssim  1 \, , $$
using $|\tau| \le |\xi|$ .

By the transfer principle we obtain
$$\| B_{\pm}^b (u,v)\|_{X^r_{0,0}} \lesssim \|u\|_{X^r_{\frac{2}{r}+b,b}} \|v\|_{X^r_{0,b}} \, .$$ 
 Here $B^b_{\pm}$ denotes the operator with Fourier symbol $b_{\pm}$ . \\
Consider now the term $||\rho|-|\eta||$ (or similarly $||\tau-\rho|-|\xi-\eta||$) in (\ref{HLR}). We have to prove
$$ \|D^{\alpha_0}((D_-^bu)v)\|_{X^r_{0,0}} \lesssim \|u\|_{X^r_{\alpha_1,b}} \|v\|_{X^r_{\alpha_2,b}} \, , $$
which is implied by
$$\|D^{\alpha_0}(uv)\|_{X^r_{0,0}} \lesssim \|u\|_{X^r_{\alpha_1,0}} \|v\|_{X^r_{\alpha_2,b}} \, . $$
After application of the fractional Leibniz rule this results from Lemma \ref{Lemma5}, because $\alpha_1+\alpha_2-\alpha_0 \ge \frac{2}{r} +b > \frac{3}{r}$ , which completes the proof. 
\end{proof}

\section{ Proof of (\ref{21}) - (\ref{37}) and Theorem \ref{Theorem1.1}:}

\begin{proof}[Proof of (\ref{21}) - (\ref{37})]
We prove the estimates first for the case $r=1+$ by the results of chapter 5. We assume $2 \ge s \ge 1+\frac{1}{r}$ , $l=1$ and $b=a=\frac{1}{r}+$ . Later we interpolate with the case $r=2$ given by Tesfahun \cite{T}. \\
{\bf Proof of (\ref{21}) and (\ref{22}):} This reduces to
$$\|q(u,v)\|_{H^r_{s-1,0}} \lesssim \|u\|_{X^r_{s,b}} \|u\|_{X^r_{s-1,b}} \, . $$
By the fractional Leibniz rule this results from Lemma \ref{Lemma1} or Lemma \ref{Lemma2} for $ s > \frac{2}{r}$. \\ 
{\bf Proof of (\ref{23}):} This reduces to
$$\|q(u,v)\|_{H^r_{l-1,0}} \lesssim \|u\|_{X^r_{s,b}} \|v\|_{X^r_{l-1,b}} \, . $$
If $l = 1$ and $s > \frac{2}{r}$ this results from Lemma \ref{Lemma1} or Lemma \ref{Lemma2}. \\
{\bf Proof of (\ref{24}) and (\ref{25}):} Concerning (\ref{24}) we have to show
$$\|q(u,v)\|_{H^r_{0,0}} \lesssim \|u\|_{X^r_{s-1,b}} \|u\|_{X^r_{s-1,b}} \, . $$
This is implied by Lemma \ref{Lemma1} or Lemma \ref{Lemma2}, if $s \ge 1+\frac{1}{r}$, which is fulfilled. Similarly (\ref{25}) follows from Lemma \ref{Lemma3}.\\
{\bf Proof of (\ref{26}):} By Lemma \ref{Lemma2.1} we have to prove
$$\|q_{ij}(u,\Lambda v)\|_{X^r_{s-1,0}} \lesssim \|u\|_{X^r_{s,b}} \|v\|_{X^r_{s,b}} \, $$
which by the fractional Leibniz rule results from Lemma \ref{Lemma1}. Moreover we need
$$\|q_{0}(u,\Lambda v)\|_{X^r_{s-1,0}} \lesssim \|u\|_{X^r_{s,b}} \|v\|_{X^r_{s,b}} \, $$
which is given by Lemma \ref{Lemma3}. Finally
$$ \| (\Lambda^{-2}u)v\|_{X^r_{s-1,0}} + \| u(\Lambda^{-2}v)\|_{X^r_{s-1,0}} \lesssim \|u\|_{X^r_{s,b}} \|v\|_{X^r_{s-1,b}}$$
by Lemma \ref{Lemma5} for $ s+2 > \frac{3}{r}$ , which is fulfilled. \\
{\bf Proof of (\ref{27}) and (\ref{28}):} The estimates result from Lemma \ref{Lemma5}, because $s+2 > \frac{3}{r}$. \\
{\bf Proof of (\ref{29}):} We reduce to
$$\|uv\|_{H^r_{s-1,0}} \lesssim \|u\|_{X^r_{l+1,b}} \|v\|_{X^r_{l,b}} \, . $$
This is implied by Cor. \ref{Cor.2} with parameters $\alpha_0 = s-1$ , $\alpha_1=l+1=2$ , $\alpha_2 =l=1$, so that $\alpha_1+\alpha_2 = 3 > \frac{3}{r}$ and $\alpha_1+\alpha_2-\alpha_0 = 4-s \ge \frac{2}{r}$ for $s \le 2$ . \\
{\bf Proof of (\ref{30}):} The estimate reduces to
$$\|uv\|_{H^r_{l-1,0}} \lesssim \|u\|_{X^r_{s+2,b}} \|v\|_{X^r_{l-1,b}} \, , $$
which results from Lemma \ref{Lemma5}, because $s+2 >\frac{3}{r}$ .\\
{\bf Proof of (\ref{31}):} We reduce to
$$\|uv\|_{H^r_{l-1,0}} \lesssim \|u\|_{X^r_{s+1,b}} \|v\|_{X^r_{s-1,b}} \, , $$
which results from Lemma \ref{Lemma5}, because $2s >\frac{3}{r}$ . \\
{\bf Proof of (\ref{32}):} The estimate
$$\|uvw\|_{H^r_{s-1,0}} \lesssim \|u\|_{X^r_{l+1,b}} \|v\|_{X^r_{s,b}}  \|w\|_{X^r_{s,b}}$$
reduces by the fractional Leibniz rule to the estimates
$$\|uvw\|_{X^r_{0,0}} \lesssim \|u\|_{X^r_{l-s+2,b}} \|v\|_{X^r_{s,b}}  \|w\|_{X^r_{s,b}}$$
and
$$\|uvw\|_{X^r_{0,0}} \lesssim \|u\|_{X^r_{l+1,b}} \|v\|_{X^r_{1,b}}  \|w\|_{X^r_{s,b}} \, .$$
Now we obtain by Lemma \ref{Lemma5} :
\begin{align*}
\|uvw\|_{X^r_{0,0}} &\lesssim \|u\|_{X^r_{l-s+2,b}} \|vw\|_{X^r_{s-l-2+\frac{3}{r}+,0}} \\
& \lesssim \|u\|_{X^r_{l-s+2,b}} \|v\|_{X^r_{s,b}}  \|w\|_{X^r_{s,b}} \, .
\end{align*}
The last estimate results from Cor. \ref{Cor.2}, if $2s-(s-l-2+\frac{3}{r}) > \frac{2}{r} \, \Leftrightarrow \, s > \frac{5}{r}-l-2= \frac{5}{r}-3$  and $2s > \frac{3}{r}$ , which is fulfilled for $s > \frac{2}{r}$ . Moreover in exactly the same way
we obtain
\begin{align*}
\|uvw\|_{X^r_{0,0}} &\lesssim \|u\|_{X^r_{l+1,b}} \|vw\|_{X^r_{\frac{3}{r}-l-1+,0}} \\
& \lesssim \|u\|_{X^r_{l+1,b}} \|v\|_{X^r_{1,b}}  \|w\|_{X^r_{s,b}} \, .
\end{align*}
{\bf Proof of (\ref{33}):} We apply Lemma \ref{Lemma5} and Cor. \ref{Cor.2} which implies
$$\|u \Lambda^{-1}(vw)\|_{H^r_{s-1,0}} \lesssim \|u\|_{X^r_{l,b}} \|\Lambda^{-1}(vw)\|_{X^r_{\frac{3}{r}+s-1-l+,0}} \lesssim \|u\|_{X^r_{l,b}} \|v\|_{X^r_{s,b}} \|w\|_{X^r_{s,b}} \, $$
for $2s- \frac{3}{r}-s+1+l+1 > \frac{2}{r} \, \Leftrightarrow s > \frac{5}{r}-3$ , which is fulfilled for $ s > \frac{2}{r}$ . \\
{\bf Proof of (\ref{35}):} Lemma \ref{Lemma5} and Cor. \ref{Cor.2} imply
$$\|u vw\|_{H^r_{s-1,0}} \lesssim \|u\|_{X^r_{s,b}} \|vw\|_{X^r_{\frac{3}{r}-1+,0}} \lesssim \|u\|_{X^r_{s,b}} \|v\|_{X^r_{s,b}} \|w\|_{X^r_{s,b}} \, , $$
because $2s-\frac{3}{r}+1 > \frac{2}{r}$ for $s > \frac{2}{r}$ , and $2s >\frac{3}{r}$ . \\
{\bf Proof of (\ref{36}):} Similarly we obtain
$$\|u vw\|_{H^r_{l-1,0}} \lesssim \|w\|_{X^r_{l,b}} \|uv\|_{X^r_{\frac{3}{r}-1+,0}} \lesssim \|w\|_{X^r_{l,b}} \|u\|_{X^r_{s,b}} \|v\|_{X^r_{s,b}} \, , $$
because $2s-\frac{3}{r}+1 > \frac{2}{r} \, \Leftrightarrow \, 2s >\frac{5}{r}-1$  , which is fulfilled for $s > \frac{2}{r}$ , and $2s > \frac{3}{r}$ .  \\
{\bf Proof of (\ref{34}):} We obtain
\begin{align*}
\|\Lambda^{-1}(uv)wz\|_{H^r_{s-1,0}} & \lesssim \|\Lambda^{-1}(uv)\|_{H^r_{\frac{3}{r}+s-1-(\frac{2}{r}+2-s)+,b}} \|wz\|_{H^r_{\frac{2}{r}+2-s,0}} \\
& \lesssim \|uv\|_{H^r_{\frac{1}{r}+2s-4+,b}} \|wz\|_{\frac{2}{r}+2-s,0} \\
& \lesssim \|u\|_{X^r_{s,b}} \|v\|_{X^r_{s,b}} \|w\|_{X^r_{s,b}} \|z\|_{X^r_{s,b}} \, .
\end{align*}
We applied Lemma \ref{Lemma5} first, then Lemma \ref{Lemma6} with parameters $\alpha_1=\alpha_2=s$, $\alpha_0 =\frac{1}{r}+2s-4$ , so that $\alpha_1+\alpha_2-\alpha_0 = 4-\frac{1}{r}  > \frac{2}{r}+b $ and $\alpha_0 \le \alpha_1=\alpha_2$ . Finally Cor. \ref{Cor.2} is used with parameters $\alpha_0= \frac{2}{r}+2-s$ , $\alpha_1=\alpha_2=s$ , so that $\alpha_0 \le \alpha_1,\alpha_2$ assuming that $s \ge 1+\frac{1}{r}$ , and also $\alpha_1+\alpha_2-\alpha_0 \ge \frac{2}{r}$ , which is equivalent to  $s \ge \frac{2}{3}+\frac{4}{3r}$ , which holds for $s \ge 1+\frac{1}{r}$ . We also obtain $\alpha_1+\alpha_2 >\frac{3}{r}$ for $s >\frac{3}{2r}$ , which is fulfilled, as well as $\alpha_1=\alpha_2 = s \neq \frac{2}{r}$ .\\
{\bf Proof of (\ref{37}):} We estimate as follows:
\begin{align*}
\|uvwz\|_{H^r_{0,0}} & \lesssim \|uv\|_{H^r_{2s-\frac{3}{r}-,b}} \|wz\|_{H^r_{\frac{6}{r}-2s+,0}} \\
& \lesssim \|u\|_{X^r_{s,b}} \|v\|_{X^r_{s,b}} \|w\|_{X^r_{s,b}}\|z\|_{X^r_{s,b}} \, .
\end{align*}
We applied Lemma \ref{Lemma5} for the first step, using $\frac{6}{r}-2s+2s-\frac{3}{r} +> \frac{3}{r}$ , and for the second step Lemma \ref{Lemma6} , using $2s-(2s-\frac{3}{r})+ > \frac{3}{r}$, and Cor. \ref{Cor.2} , using $2s-\frac{6}{r}+2s > \frac{2}{r} \, \Leftrightarrow \, s > \frac{2}{r}$ .
\end{proof}

\begin{proof}[Proof of Theorem \ref{Theorem1.1}]
Now we recall the results of Tesfahun \cite{T}, who proved the same estimates in the case $r=2$ with $s=\frac{6}{7}+\epsilon$ , $l=-\frac{1}{14}+\epsilon$ , $a=\half+$ , $b= \frac{3}{4}+$ for $ \epsilon > 0$ . Thus we may apply bilinear interpolation between the case $r=1+$ and $r=2$ . 

Let $\delta > 0$ be given and $s=\frac{16}{7r}-\frac{2}{7}+\delta$ , $ l=\frac{15}{7r} -\frac{8}{7}+\delta$ . Then for $r>1$ sufficiently close to 1 we have $\delta \ge \frac{15}{7}-\frac{15}{7r} + \omega$ with $\omega \ge 0$ . In the case $\omega=0$ we have $s= \frac{13}{7}+\frac{1}{7r}$ and $l=1$ . Because in this case $ s \ge 1+\frac{1}{r}$ , we have proven that the estimates (\ref{21}) - (\ref{37}) are true. By the fractional Leibniz rule they remain true for $\omega > 0$ , thus for the given $\delta$ and $r > 1$ close enough to 1.
Bilinear interpolation between this case and $r=2$ implies the estimates for the whole range $ 1<r \le2$. This gives the estimates (\ref{21}) - (\ref{37}) for the following choice of parameters for the whole range $1<r \le 2$ :
$$ s = \frac{16}{7r}-\frac{2}{7}+\delta \,,\quad l= \frac{15}{7r} - \frac{8}{7}+\delta \,,\quad a= \frac{1}{r}+ \,,\quad b=\half + \frac{1}{2r}+ \, . $$
Finally, an application of Theorem \ref{Theorem0.3} to the Cauchy problem (\ref{YMF2}),(\ref{YM4}),(\ref{Data-AF}) completes the proof of Theorem \ref{Theorem1.1}.
\end{proof}

\end{document}